\documentclass{amsart}
\usepackage{latexsym}
\usepackage{amsmath,amsthm,amssymb, amscd,color}
\usepackage{etex, verbatim}
\usepackage[all]{xy}

\usepackage{graphicx}
\usepackage{color}
\usepackage[pagewise]{lineno}
\usepackage{perpage} 
\MakePerPage{footnote} 
\usepackage{arydshln}
\input xy \xyoption{all}
\usepackage{blkarray} 

\usepackage[utf8]{inputenc}
\usepackage[normalem]{ulem}

\theoremstyle{plain}
\newtheorem{theorem}{Theorem}[section]

\newtheorem{lemma}[theorem]{Lemma}
\newtheorem{corollary}[theorem]{Corollary}
\numberwithin{equation}{section}

\theoremstyle{definition}
\newtheorem{definition}[theorem]{Definition}
\newtheorem{example}[theorem]{Example}

\newtheorem{proposition}[theorem]{Proposition}

\newtheorem{remark}[theorem]{Remark}

\newcommand{\R}{\mathbb{R}}
\newcommand{\Z}{\mathbb{Z}}

\newcommand{\lra}{\longrightarrow}

\makeatletter \def\blfootnote{\xdef\@thefnmark{}\@footnotetext}\makeatother

\numberwithin{equation}{section}

\def\quasitoric{toric }

\def \({\left(}
\def \){\right)}
\def \<{\langle}
\def \>{\rangle}

\def \tensor{\otimes}

\def \bb{\mathbb}

\def \RR{{\bb{R}}}

\def \ZZ{{\bb{Z}}}

\def \begineq{\begin{equation}}
\def \endeq{\end{equation}}

\numberwithin{equation}{section}

\begin{document}

\title[Equivariant K-theory of toric orbifolds]{Equivariant K-theory
  of toric orbifolds}

\author[S. Sarkar]{Soumen Sarkar}
\address{Department of Mathematics, Indian Institute of Technology Madras, India}
\email{soumensarkar20@gmail.com}

\author[V. Uma]{V. Uma}
\address{Department of Mathematics, Indian Institute of Technology Madras, India}
\email{vuma@iitm.ac.in}

%

\subjclass[2010]{Primary 55N15, 55N22, 55N91, 14M25; Secondary 55N10, 52B11}

\keywords{toric orbifold, quasitoric orbifold, toric variety, projective toric variety, 
piecewise polynomial, piecewise Laurent polynomial}

\date{\today}
\dedicatory{}
\maketitle

\begin{abstract}
  Toric orbifolds are a topological generalization of projective toric
  varieties associated to simplicial fans.  We introduce some
  sufficient conditions on the combinatorial data associated to a
  toric orbifold to ensure the existence of an invariant cell
  structure on it and call such a toric orbifold {\it retractable}. In
  this paper, our main goal is to study equivariant cohomology
  theories of retractable toric orbifolds. Our results extend the
  corresponding results on {\it divisive} weighted projective spaces.
 \end{abstract}

\tableofcontents

\section{Introduction}
The notion of a {\it toric orbifold} was introduced by Davis and
Januskiewicz in \cite[Section 7]{DJ}. It was later studied by Poddar
and the first author in \cite{PS} who call it a ``quasitoric
orbifold''. Loosely speaking a toric orbifold $X$ is an orbifold
admitting an effective action by a compact torus
$T\cong (\mathbb{S}^1)^n$ with orbit space a simple convex polytope
$Q$. A toric orbifold can alternately be constructed from $Q$ and the
data encoded by a $\mathbb{Z}^n$-valued function $\lambda$ on the set
of codimension-one faces of $Q$ (see Definition
\ref{def_characteristic_function}). In particular, when $\lambda$
satisfies the $(*)$ condition in \cite[p. 423]{DJ}, the toric orbifold
is smooth and is called a {\it toric manifold}, also known as a
``quasitoric manifold'' (see \cite[Section 5.2]{BP}). A toric manifold
is known to admit a canonical $T$-invariant cell structure given by
means of a height function on $Q$ \cite[Theorem 3.1]{DJ}. A torus
manifold is an even-dimensional manifold acted on by a
half-dimensional torus with non-empty fixed point set and some
additional orientation data (see \cite{hm, mp}). The orbit space of a
torus manifold has a rich combinatorial structure, e.g., it is a
manifold with corners provided that the action is locally
standard. These are generalizations of toric manifolds. But unlike a
toric manifold, in general, a torus manifold does not come equipped
with an invariant cellular structure, unless we impose some additional
combinatorial conditions on it (see for example \cite[Lemma
2.2]{u}).

In this paper, we give some sufficient conditions for a possibly
singular toric orbifold to admit a $T$-invariant cell structure. We do
this by using the concept of retraction sequence of the simple
polytope $Q$ which was introduced in \cite{BSS}. We call a toric
orbifold satisfying the above sufficiency condition a {\it retractable
  toric orbifold}. Our definition of a retractable toric orbifold was
motivated by the invariant cell structure on a {\it divisive} weighted
projective space described by Harada, Holm, Ray and Williams in
\cite[Proposition 2.7, Corollary 2.9]{HHRW}. Indeed, the divisive
weighted projective space is a particular example of a retractable toric
orbifold which has the simplex as the quotient polytope (see Example
\ref{dwp}).

Let $pt$ denote the $1$-point space with the trivial $T$-action.  The
equivariant projection $X\longrightarrow pt$ induces the structure of
a graded $E_{T}^*(pt)$-algebra structure on $E_{T}^*(X)$, where
$E_{T}^*$ is any generalized $T$-equivariant cohomology theory. We
describe $E_T^*(X)$ as an algebra over $E_T^*(pt)$ when $X$ is a
retractable toric orbifold (see Corollary \ref{gkm_rt_algebra}). In
particular, we get a description of the $T$-equivariant $K$-theory
ring of $X$ over $K_{T}^*(pt)\cong R(T)[z,z^{-1}]$, where
$R(T)=K_{T}^0(pt)$ denotes the ring of finite dimensional complex
representations of $T$ and $z$ denotes the Bott periodicity element in
$K^{-2}(pt)$. This result generalizes the corresponding result on a
divisive weighted projective space in \cite[Proposition
3.10]{HHRW}. As in \cite{HHRW}, our main tools here are the methods
developed by Harada, Henriques and Holm in \cite{HHH}.  For a
topological group $G$, they prove a GKM-type theorem for the
$G$-equivariant generalized cohomology theory of a $G$-space equipped
with a $G$-invariant stratification satisfying some additional
conditions \cite[Section 3]{HHH}. In Proposition \ref{divass}, we show
that a retractable toric orbifold $X$ has a $T$-invariant stratification
satisfying these conditions.

In Section 4, we introduce the notion of \emph{piecewise algebra} for a
characteristic pair $(Q, \lambda)$, see Definition \ref{piec_alg}, in
a dual manner to the concept of piecewise algebra associated to a fan
(see \cite[Section 4]{HHRW}). In our main result, Theorem
\ref{piecewise}, we prove that the generalized equivariant cohomology
theory ring of a retractable toric orbifold is isomorphic as an
$E_T^*(pt)$-algebra to the corresponding piecewise algebra for
$(Q, \lambda)$. This result extends the corresponding result for a
divisive weighted projective space in \cite[Theorem 5.5]{HHRW} to the
larger class of retractable toric orbifolds (see Example
\ref{eg_div_toric_3d}).

In particular, we show that the equivariant integral cohomology ring
$H_{T}^*(X)$ is isomorphic to the piecewise polynomial functions on
$(Q, \lambda)$, the equivariant topological $K$-theory ring
$K_{T}^*(X)$ is isomorphic to the piecewise Laurent polynomial
functions on $(Q,\lambda)$, and the equivariant complex cobordism ring
$MU_{T}^*(X)$ is isomorphic to piecewise cobordism forms on
$(Q,\lambda)$.

Note that examples of toric orbifolds include simplicial projective
toric varieties which correspond to simplicial polytopal fans. Thus
Corollary \ref{gkm_rt_algebra} and Theorem \ref{piecewise} give
generalizations of the corresponding results on the toric variety
associated to a smooth polytopal fan in \cite[Theorem 7.1]{HHRW} and
\cite[Corollary 7.2]{HHRW} respectively, to toric varieties associated
to simplicial polytopal fans which have the underlying structure of a
retractable toric orbifold under the action of the compact torus.

We refer to \cite{May} for the definitions and results on $T$-equivariant
generalized cohomology theories $E_{T}^*$, \cite{Segal} for
$T$-equivariant $K$-theory $K_{T}^*$ and \cite{TD} and \cite{Sinha}
for $T$-equivariant complex cobordism theory $MU^*_{T}$.

The paper is organized as follows. We recall some basics of toric
orbifolds, the concept of local groups, and the concept of retraction
of simple polytopes in Sections
\ref{subsec_def_by_construction}--\ref{subsec_char_subsp},
\ref{subsec_def_local_groups}, and \ref{sec_ret_of_simple_poly}
respectively.  We introduce the notion of ``retractable toric orbifolds''
in Section \ref{subsec_retractable_toric}. We discuss GKM-theory of
retractable toric orbifolds in Section \ref{sec_gkm-theory_toric_orb}
(see Proposition \ref{divass} and Corollary \ref{gkm_rt_algebra}). The
concept of ``piecewise algebras'' for a characteristic pair is
introduced in Section \ref{sec_piecewise_alg}.  We prove our main
result in Theorem \ref{piecewise}.
 
\section{Cell-structure of \quasitoric orbifolds}\label{cw-structure}
In this section we briefly recall the concept of a
\emph{characteristic pair} $(Q,\lambda)$ from \cite{DJ} and \cite{PS},
and explain how it is used to construct a \emph{\quasitoric orbifold}
$X:=X(Q,\lambda)$. The main goal of this section is to introduce some
sufficient conditions on a toric orbifold $X$ which is singular to
have a cell structure. To complete this goal, we recall three
additional concepts, namely the \emph{characteristic subspaces} of
$X$, the {\it local groups} corresponding to a face and a vertex of
it, and the {\it retraction} of the simple polytope $Q$.


\subsection{Definition of toric orbifolds by construction}\label{subsec_def_by_construction}
In this subsection, we briefly recall the constructive definition of (quasi)toric 
orbifolds following \cite{DJ} and \cite{PS}. Let $M$ be a submodule of
$\ZZ^n \subset \RR^n$ over  $\Z$, $M_{\RR}:= M \tensor_{\ZZ} \RR$ and
 $T_M := M_{\RR}/M$. Then we have the natural inclusions 
$f \colon M_{\RR} \to \ZZ^n \otimes_\Z \RR = \RR^n$ 
and $f_{\ast}\colon T_{M}  \to \RR^n/M$. Note that the inclusion
$i \colon M \to \ZZ^n$ induces a group homomorphism  
$$i_\ast \colon (\Z^n \otimes_\Z \R) / M \to (\Z^n \otimes_\Z \R) / \Z^n,$$
defined by $i_{\ast}(a + M) = a + \ZZ^n$. Then Ker$(i_{\ast}) \cong \ZZ^n/M$. 
The range space is the $n$-dimensional standard torus. We denote this torus by $T$. 
Let $f_{M}$ be the composition $i_{\ast} \circ f_{\ast}\colon T_{M} \to T$.
 If the rank of $M$ is $n$, then the map $f_{M} \colon T_{M} \to T$ is a surjective
homomorphism with kernel $G_{M} =\ZZ^n/M$, a finite abelian group.

Let $Q$ be an $n$-dimensional simple convex polytope in $\RR^n$ and 
$$\mathcal{F}(Q)= \{F_i : i \in \{ 1, \ldots, d \} = I\}$$ 
be the codimension-one faces (facets) of $Q$.
\begin{definition}\label{def_characteristic_function}
  A function $\lambda \colon \mathcal{F}(Q) \to \ZZ^n$ is called a
  {\em characteristic function} on $Q$ if
  $\lambda(F_{i_1}), \ldots, \lambda(F_{i_k})$ are linearly
  independent primitive vectors whenever the intersection
  $F_{i_1} \cap \cdots \cap F_{i_k}$ is nonempty. Then
  $\lambda_i := \lambda(F_i)$ is called the {\em characteristic vector}
  corresponding to the facet $F_i$.  The pair $(Q, \lambda)$ is called
  {\em a characteristic pair}.
\end{definition}

We remark here that in the above definition it suffices for $\lambda$
to satisfy the linear independency at each vertex which is an
intersection of $n$ facets. An example of a characteristic function is
given in Figure \ref{fig-eg1}.

Let $F$ be a codimension-$k$ face of $Q$. Since $Q$ is a simple
polytope, $F$ is the unique intersection of $k$ facets
$F_{i_1}, \ldots, F_{i_k}$.  Let $M(F)$ be the submodule of $\ZZ^n$
generated by the characteristic vectors
$\{\lambda_{i_1}, \dots, \lambda_{i_k} \}$. Then,
$T_{M(F)} = M(F)_{\RR} /M(F)$ is a torus of dimension $k$. We shall
adopt the convention that $T_{M(Q)} = 1$.  Let
\begin{equation}\label{def_tf}
T_F = \text{Im}\{f_{M(F)} \colon T_{M(F)} \to T\}
\end{equation}
Define an equivalence relation $\sim$ on the product $T \times Q$ by
\begin{equation}\label{equ001}
 	(t, x) \sim (s, y) ~ \mbox{if and only if}~ 
	 x=y~ \mbox{and}~ s^{-1}t \in T_F     
\end{equation}
where $F$ is the smallest face containing $x$. The quotient space 
$$ X(Q, \lambda)=(T \times Q)/\sim$$
has an orbifold structure with a natural $T$-action induced by 
the group operation, see Section 2 in \cite{PS}. Clearly, the orbit space
of $T$-action on $X(Q, \lambda)$ is $Q$. Let 
$$\pi \colon X(Q, \lambda) \to Q ~ \mbox{defined by} ~  \pi([p,t]_{\sim}) = p$$
be the orbit map. The space $X(Q, \lambda)$ is called the {\em (quasi)toric orbifold} 
associated to the characteristic pair $(Q, \lambda)$. 

We note that if, in addition, $\lambda$ satisfies the Davis and
Januszkiewicz's condition $(*)$ in \cite[p. 423]{DJ}, namely that the
primitive vectors $\lambda(F_{i_1}), \ldots, \lambda(F_{i_k})$ can be
extended to form a basis of $\mathbb{Z}^n$ whenever the intersection
$F_{i_1} \cap \cdots \cap F_{i_k}$ is nonempty, then $X$ is a smooth
manifold called a {\em (quasi)toric manifold}.

After analyzing the orbifold structure of $X(Q, \lambda)$, in
\cite[Subsection $2.2$]{PS}, Poddar and the first author also gave an
axiomatic definition of (quasi)toric orbifolds, which generalizes the
axiomatic definition of toric manifolds of \cite[Section 1]{DJ}.  In
\cite[Section 2]{PS}, the authors give explicit orbifold charts (in
the sense of \cite[Section 1.1]{ALR}) of $X(Q, \lambda)$.

\begin{remark}
  One can alternately construct a toric orbifold as the quotient of a
  {\em moment angle complex} by the action of a torus determined by
  the characteristic function, see \cite[Chapter 6]{BP} for the
  arguments.
\end{remark}


\subsection{Invariant and characteristic subspaces}\label{subsec_char_subsp}
We now describe certain closed invariant subspaces of a toric
orbifold $X(Q, \lambda)$. Let $F$ be a face of $Q$ of codimension
$k$. Then, the pre-image $\pi^{-1}(F)$ is a closed invariant
subspace. Indeed with the subspace topology, $\pi^{-1}(F)$ is a toric
orbifold of dimension $2n-2k$.  The corresponding characteristic pair
for $\pi^{-1}(F)$ can be described as follows.

Let
\begin{equation}
M^*(F)= M(F)_\R \cap \Z^n ~~{\rm and}~~ G_F = M^*(F)/M(F).
\end{equation}
Here, $M(F) \subseteq M^*(F)$ and both are free $\Z$-modules of rank $k$,
therefore $G_F$ is a finite abelian group. Note that if $F$ is a face of $F'$,
 then the natural inclusion of $M(F')$ into $M(F)$ induces a surjective 
 homomorphism from $G_{F'}$ to $G_{F}$. Moreover, since $M^\ast(F)$ is a
 free $\Z$-module of rank $k$, one may identify $M^\ast(F)$ with $\Z^k$
 by fixing a proper isomorphism.
 
Consider the following projection homomorphism$\colon$
\begin{equation}\label{eq_lattice_projection}
	\varrho_F \colon \ZZ^n \to  \Z^n / M^\ast(F)\cong \Z^{n-k}.
\end{equation}
Let $\{H_{1}, \ldots, H_{\ell}\}$ be the facets of $F$.  Then for each
$j \in \{1, \ldots, \ell\}$, there is a unique facet $F_{i_j}$ of $Q$
such that $H_j = F \cap F_{i_j}$. We define a map
\begin{equation}\label{charface}\lambda_{F} \colon \{H_1, \ldots, H_\ell \} \to \Z^{n-k} \text{ by }\end{equation}  
$$\lambda_{F}(H_j) =\mathit{prim}(\varrho_F \circ \lambda(F_{i_j}))
~\mbox{for}~j \in \{1, \ldots, \ell  \},$$
where $\mathit{prim}(x)$ denotes the primitive vector of $x$ in
$\ZZ^{n-k}$.  Then $\lambda_F$ is a characteristic
function on $F$.  Let $X(F, \lambda_F)$ be the toric orbifold
corresponding to the characteristic pair $(F, \lambda_F)$. Then
$$X(F,\lambda_F)=(T^{n-k}\times F) \big/ \sim_{\lambda_{F}},$$
where the equivalence relation $\sim_{\lambda_{F}}$ is defined similar
to \eqref{equ001}. Further, by \cite[Proposition 3.2]{BSS}, as
topological spaces, $\pi^{-1}(F)$ and $X(F, \lambda_F)$ are
homeomorphic.

In particular, if $F_i$ is a facet of $Q$, the subspace $\pi^{-1}(F_i)$ is
called a {\em characteristic subspace} of $X(Q, \lambda)$.

\subsection{Toric orbifolds and local groups}\label{subsec_def_local_groups}
In this subsection, we define the {\em local groups} $G_{F}(v)$ of
$X(Q, \lambda)$ corresponding to a codimension-$k$ face $F$ of $Q$
containing a vertex $v$. There is a bijective
correspondence between the fixed points of the $T^{n-k}$-action on the
toric orbifold $\pi^{-1}(F)$ and the vertices of $F$.  Let $v$
be a vertex of $Q$. Then $v=F_{i_1} \cap \cdots \cap F_{i_n}$ for a
unique collection of facets $F_{i_1}, \ldots, F_{i_n}$ of $Q$.  Let
$M(v)$ be the submodule of $\ZZ^n$ generated by
$\{\lambda(F_{i_1}), \ldots, \lambda(F_{i_n})\}$ where $\lambda$ is
as in Definition \ref{def_characteristic_function}.  Also, we have
$v=H_{j_1} \cap \cdots \cap H_{j_{n-k}}$ for a unique collection of
facets $H_{j_1}, \ldots, H_{j_{n-k}}$ of $F$. Let $M_{F}(v)$ be the
submodule of $ \Z^{n-k}$ generated by
$\{\lambda_{F}(H_{j_1}), \ldots, \lambda_{F}(H_{j_{n-k}})\}$ where
$\lambda_{F}$ is defined in (\ref{charface}).  We define
\begin{align*}
  G_{Q}(v)&:=\ZZ^n/M(v), \\
  G_{F}(v)&:= \Z^{n-k} /M_{F}(v).
\end{align*}
These are finite abelian groups. Notice that the orders $|G_Q(v)|$ and
 $|G_F(v)|$ of each group
are obtained by computing the corresponding determinant. More precisely, 
\begin{align*} 
  |G_Q(v)|&=\bigm|\det\left[\begin{array}{c|c|c} \lambda(F_{i_1})
                              &\cdots&\lambda(F_{i_n}) \end{array} 
                                       \right ]\bigm|, 
  \\
  |G_F(v)|&=\bigm|\det\left[ \begin{array}{c|c|c} \lambda_F(H_{j_1})
                               &\cdots&\lambda_F(H_{j_{n-k}}) \end{array}
                                        \right]\bigm|.\end{align*} In
                                      particular, for a vertex $v$ of
                                      $Q$, $G_v(v)=\{1\}$.


                                      \subsection{Retraction sequences
                                        of simple
                                        polytopes}\label{sec_ret_of_simple_poly}

In this subsection, we recall the concept of {\it retraction sequence}
of a simple polytope which was introduced in \cite{BSS}.
  
Let $Q$ be an $n$-dimensional simple polytope with $m$ vertices.  We
now construct a sequence of triples
$\{(B_{\ell}, P_{\ell}, v_{\ell})\}$. Let $B_1 = Q=P_1$ and
$v_1 \in V(B_1)$. Suppose $(B_\ell,P_\ell,v_\ell)$ has been defined
for $1\leq \ell\leq k-1$. We define $(B_{k}, P_{k},v_{k})$ inductively
as follows:
\begin{itemize}
\item Let $B_{k}$ denote the subset of $B_{k-1}$ such that
\[ B_{k}:=\bigcup \{F~\mid~F~~\mbox{is a face of} ~~Q ~~\mbox{contained
    in} ~~B_{k-1}~~\mbox{and}~~v_{k-1}\notin V(F)\}.\] 
\item Let $v_{k}\in V(B_{k})$ be such that $v_{k}$ has a
neighbourhood $U_{k}$ in $B_{k}$ which is homeomorphic to
$\RR^{s_{k}}_{\geq}$ as a manifold with corners for some
$0 \leq s_{k} \leq \dim(B_k)$.

\item Let $P_{k}$ be the smallest face of $B_{k}$ containing
  $U_{k}$. Hence $U_{k}$ is obtained from $P_{k}$ by deleting all its
  faces which do not contain $v_{k}$.
\end{itemize}

The sequence $\{(B_{\ell}, P_{\ell}, v_{\ell})\} $ stops if there is no
vertex of $B_{\ell}$ which has a neighbourhood in $B_{\ell}$ that is
homeomorphic to $\RR^{s_{\ell}}_{\geq}$ as a manifold with corners
for some $0 \leq s_{\ell} \leq \dim(B_{\ell})$.  Proceeding in this
way, if we get that $B_m$ is the vertex $v_m$, we set
$P_m=\{v_m\}=B_m$. At this point we introduce the following
definition.

\begin{definition}
  Let $Q$ be a simple polytope. If there exists a sequence
  $\{(B_{\ell}, P_{\ell}, v_{\ell})\}_{\ell=1}^m $ as constructed above
  such that $B_m=P_m=v_m$ is a vertex, then we say that
  $\{(B_{\ell}, P_{\ell}, v_{\ell})\}_{\ell = 1}^m$ is a {\em retraction
    sequence} of $Q$ starting with the vertex $v_1$ and ending at
  $v_m$.
\end{definition}
In relation to the above, we also recall the following definition.

\begin{definition}
A vertex $v$ is called a {\em free vertex}  of $B_{\ell}$ if it has a neighbourhood
in $B_{\ell}$ which is homeomorphic to $\RR^s_{\geq}$ as manifold with corners
for some $0 \leq s \leq \dim(B_{\ell})$.
\end{definition}

\begin{remark}
\begin{enumerate}
\item In the retraction sequence, a choice of a free vertex $v_\ell$
    in $B_{\ell}$ determines $B_{\ell+1}$.
    
\item In a retraction sequence of a simple polytope $Q$, the number
    of retraction steps is $m= |V(Q)|$. In particular, it is
    independent of the choice of the retraction sequence.

\item By \cite[Proposition 2.3]{BSS}, the height function on a simple
  polytope $Q$ gives a retraction sequence of it.

  \item See Figure \ref{fig-eg5} for a retraction sequence of a $3$-prism.
  
\end{enumerate}
\end{remark}

\begin{remark}\label{dir_graph}
  Given a retraction sequence of simple polytope $Q$ one can define a
   directed graph on the 1-skeleton of $Q$ in the following
  way. Let $\{(B_i, P_i, v_i)\}_{i=1}^m$ be a retraction sequence of
  $Q$. We order the vertices of $Q$ as $v_1 < v_2 < \ldots < v_m$.  We
  assign a direction from $v_s$ to $v_r$ if there is an edge with end
  points $v_s, v_r$ and $v_s > v_r$. This directed graph has the
  property that if $s_i$ number of edges end at the vertex $v_i$ then
  $\dim P_i =s_i$.
\end{remark}



\subsection{Torus invariant cell structures on toric orbifolds}\label{subsec_retractable_toric}
In this subsection, we first define {\em retractable} toric orbifolds,
and then show the existence of a {\em genuine} invariant cell
structure on a retractable toric orbifold justifying the
terminology. (Here genuine is as opposed to the existence of a {\bf
  q}-cellular structure on any toric orbifold by \cite[Section
4]{PS}). We adhere to the notation of the previous subsections.

\begin{definition}\label{def_divisive_tor_orb} 
  Let $X$ be a toric orbifold over the simple polytope $Q$. Then
  $X$ is called {\em retractable} if $Q$ has a retraction
  $\{(B_i, P_i, v_i)\}_{i=1}^m$ such that $G_{P_i}(v_i)$ is the
  trivial group for $i=1, \ldots, m-1$.
\end{definition}

\begin{lemma}\label{cw-of-toic-orbi}
  If $X$ is a {\em retractable} toric orbifold of dimension $2n$, then $X$ has
  a cell structure with $T$-invariant cells.
\end{lemma}
\begin{proof}
  Let $\{H_1, \ldots, H_{s_i}\}$ be the facets of $P_i$ such that
  $H_1 \cap \cdots \cap H_{s_i} = v_i$. Let $U_i\simeq
  \mathbb{R}_{\geq}^{s_i}$ be the open neighbourhood of $v_i$ in
  $P_i$ obtained by deleting all faces of $P_i$ not containing the
  free vertex $v_i$. Since $G_{P_i}(v_i)$ is trivial, the collection
  of vectors $\{\lambda_{P_i}(H_1), \ldots, \lambda_{P_i}(H_{s_i})\}$
  form a basis of $\ZZ^{s_i}$. Now, following the arguments in
  \cite[Section 1.5 and Lemma 1.6]{DJ}, we have that
  $$(T^{s_i} \times U_i)/\sim _{\lambda_{P_i}}~= ~ \pi_{P_i}^{-1}(U_i)$$
  can be identified with $\mathbb{C}^{s_i}$ having the standard
  $T^{s_i}$ action. Alternately by \cite[Section 4.2]{PS},
  $\pi_{P_i}^{-1}(U_i)$ is equivariantly homeomorphic to the quotient
  of a $2s_i$-dimensional open disk in $\mathbb{R}^{2s_i}$ by the
  group $G_{P_i}(v_i)$ which is trivial by assumption, and hence
  equivariantly homeomorphic to $\mathbb{C}^{s_i}$. Now by
  \cite[Proposition 3.2]{BSS},
  $\pi_{P_i}^{-1}(U_i) \subset X(P_i, \lambda_{P_i})$ is equivariantly
  homeomorphic to $\widehat{U}_i:=\pi^{-1}(U_i) \subset X(Q, \lambda)$
  for $i = 1, \ldots, m$. Further, note that
  $Q= {\displaystyle \bigcup_{i=1}^m U_i}$. Therefore
  $X(Q, \lambda) = {\displaystyle \bigcup_{i=1}^m \widehat{U}_i}$.
  This proves the lemma.
\end{proof}

\begin{example}\label{dwp}
  The {\it divisive} weighted projective spaces of \cite[Definition
  2.2]{HHRW} are examples of {\em retractable} toric orbifolds over an
  $n$-simplex $Q$. Let $F_1,\ldots, F_{n+1}$ denote the facets of
  $Q$. The characteristic function $\lambda$ is given by
  $\lambda(F_i)=e_i$ for $1\leq i\leq n$ and
  $\lambda(F_{n+1})=(-\chi_1,-\chi_2,\ldots, -\chi_n)$ where
  $\chi_1=1$ and $\chi_i$ divides $\chi_{i+1}$ for $1\leq i\leq n-1$
  (divisive condition). Let $v_0$ denote the intersection of the
  facets $F_1,\ldots, F_n$. For $1\leq i\leq n$, let $v_i$ denote the
  vertex of $Q$ which is the intersection of the facets $F_j$ for
  $1\leq j\leq n+1$ such that $j\neq i$. The ordering of the vertices
  $v_{0}<v_1<v_2<\cdots<v_{n}$ of $Q$ by means of a generic height
  function gives a retraction sequence
  $\{(B_i,P_i,v_i)\}_{i=0}^{n}$. One can check using the divisive
  condition that the local groups satisfy $G_{P_i}(v_i)=1$ for every
  $0\leq i\leq n$. For details of this computation see \cite{Sar}.
\end{example}

\begin{example}\label{eg_div_toric_3d}
  We now give an example of a retractable toric orbifold which is not a
  divisive weighted projective space.  Consider the characteristic
  function $\lambda$ on a 3-prism $Q$ as in Figure \ref{fig-eg1}. One
  can compute that $G_Q(v_{1})=\{1\}$,
  $G_Q(v_{2})=\mathbb{Z}/3\mathbb{Z}$, $G_Q(v_{3})=\{1\}$,
  $G_Q(v_{4})=\{1\}$, $G_Q(v_{5})=\mathbb{Z}/21\mathbb{Z}$, $G_Q(v_{6})=\mathbb{Z}/7\mathbb{Z}$, where, for
  example, $G_Q(v_{1})$ is the determinant of
  $\{\lambda(F_1), \lambda(F_2), \lambda(F_3)\}$, and similarly for
  the other groups.
 
\begin{figure}[ht]
        \centerline{
           \scalebox{0.75}{
            \input{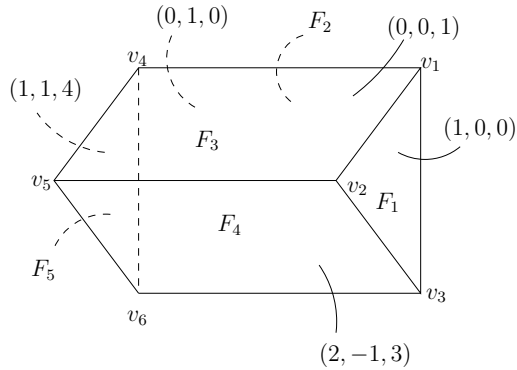}
            }
          }
 \caption{A characteristic function on 3-prism $Q$.}
 \label{fig-eg1}
 \end{figure} 
 
 Now we consider the following retraction of $Q$, see Figure
 \ref{fig-eg5}.  The retraction sequence is given by
 $(B_1, B_1, v_{1})$, $~(B_2, F_4, v_{2})$,
 $~(B_3, F_3 \cap F_4, v_{3})$, $(B_4, F_5, v_{4})$,
 $(B_5, F_4 \cap F_5, v_{5})$, $(B_6, v_6, v_{6})$.  We only compute
 the local group $G_{F_4}(v_{2})$ and the computation for other
 $G_{P_i}(v_i)$ is similar. In this case $P_i = F_4$ and $v_i = v_2$.
 So $M(F_4) = \<\lambda(F_4)\>= \<( 6, -1, 3)\>$.  Thus
 $$M^*(F_4) = \<\lambda(F_5)\> \cap \ZZ = \<(6, -1, 3)\> = M(F_4) \cong
 \ZZ.$$ Consider the basis $\{e_1=(1,0,0), e_2=(0,0,1), e_3=(6,-1,3)\}$
 of $\ZZ^3$.  Then one gets the projection
 $$\rho \colon \Z^3 \to \Z^3/M^*(F_4) = \<e_1, e_2, e_3\>/\<e_3\> \cong
 \ZZ^2.$$ The facets of $F_4$ which intersects at $v_{2}$ are
 $F_1 \cap F_4$ and $F_2 \cap F_4$. Therefore we get $\lambda_{F_4}$
 on $F_1 \cap F_4$ and $F_2 \cap F_4$ which are given by
$$\lambda_{F_4}(F_1 \cap F_4) = \mathit{prim}(\rho(\lambda(F_1)))=(1, 0)$$  and
$$\lambda_{F_4}(F_2 \cap F_4) = \mathit{prim}(\rho(\lambda(F_2)))=prim(6, 3)=(2,1).$$
Therefore
$$G_{F_4}(v_{2})= \Z^2/\<\lambda_{F_4}(F_1 \cap F_4), \lambda_{F_4}(F_2 \cap F_4)\>
=\Z^2/\<(1, 0), (2, 1)\>=\{1\}.$$

Similarly, one can compute that $G_{F_3 \cap F_4} (v_{3})=\{1\}$,
$G_{F_5}(v_{4})=\{1\}$, $G_{F_4\cap F_5}(v_{5})=\{1\}$.

\begin{figure}[ht]
        \centerline{
           \scalebox{0.45}{
            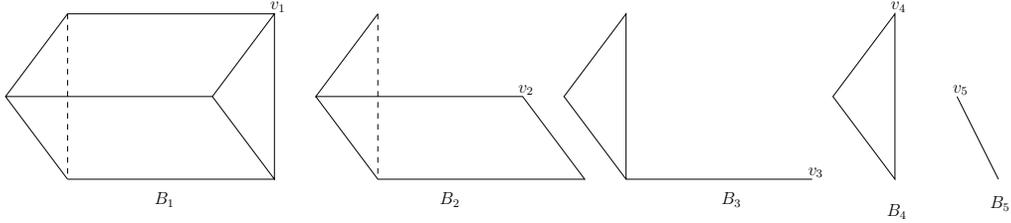
            }
          }
 \caption{A retraction sequence of 3-prism $Q$.}
 \label{fig-eg5}
 \end{figure}
 \end{example}

 In the above example if instead of $\lambda(F_2)=(6,-1,3)$ we let
 $\lambda(F_2)=(2,-1,3)$ then the local groups would be
 $G_Q(v_{1})=\{1\}$, $G_Q(v_{2})=\mathbb{Z}/3\mathbb{Z}$,
 $G_Q(v_{3})=\{1\}$, $G_Q(v_{4})=\{1\}$,
 $G_Q(v_{5})=\mathbb{Z}/5\mathbb{Z}$,
 $G_Q(v_{6})=\mathbb{Z}/3\mathbb{Z}$.

 For the retraction of $Q$, given in Figure \ref{fig-eg5} and the
 corresponding retraction sequence  described above we check that the
 local group
 $G_{F_4}(v_{2})$ is not $\{1\}$. In this case we have
 $M(F_4) = \<\lambda(F_4)\>= \<( 2, -1, 3)\>$.  Thus
 $$M^*(F_4) = \<\lambda(F_5)\> \cap \ZZ = \<(2, -1, 3)\> = M(F_4) \cong
 \ZZ.$$ Consider the basis $\{e_1=(1,0,0), e_2=(0,0,1), e_3=(2,-1,3)\}$
 of $\ZZ^3$.  Then one gets the projection
 $$\rho \colon \Z^3 \to \Z^3/M^*(F_4) = \<e_1, e_2, e_3\>/\<e_3\> \cong
 \ZZ^2.$$ The facets of $F_4$ which intersects at $v_{2}$ are
 $F_1 \cap F_4$ and $F_2 \cap F_4$. Therefore we get $\lambda_{F_4}$
 on $F_1 \cap F_4$ and $F_2 \cap F_4$ which are given by
$$\lambda_{F_4}(F_1 \cap F_4) = \mathit{prim}(\rho(\lambda(F_1)))=(1, 0)$$  and
$$\lambda_{F_4}(F_2 \cap F_4) = \mathit{prim}(\rho(\lambda(F_2)))=prim(2, 3)=(2,3).$$
Therefore
$$G_{F_4}(v_{2})= \Z^2/\<\lambda_{F_4}(F_1 \cap F_4), \lambda_{F_4}(F_2 \cap F_4)\>
=\Z^2/\<(1, 0), (2, 3)\>=\mathbb{Z}/3\mathbb{Z}$$

 \begin{remark}
   Note that if one consider the retraction sequence
   $\{(B_i, P_i, v_i)\}_{i=1}^6$ as in Figure \ref{fig-kthm3}, then it
   induces a $T$-invariant CW-structure on $X(Q, \lambda)$.  Therefore
   \cite[Proposition 4.3]{HHH} can be applied to get explicit
   generators of $H^{\ast}_T(X(Q, \lambda))$ as a
   $H^{\ast}_T(pt)$-module. More generally, let $(Q, \lambda)$ be a
   characteristic pair and let $\{(B_i, P_i, v_i)\}_{i=1}^m$ be a
   retraction sequence of $Q$ obtained from a height function on $Q$
   such that $G_{P_i}(v_i)=1$ for $i=1, \ldots, m$. Then it can be
   seen that the numbers $s_i$'s in the definition of the retraction
   sequence are non-decreasing. Thus the induced $T$-invariant
   retractable structure on the toric orbifold $X(Q,\lambda)$ is in fact
   a $T$-invariant CW structure.  Consequently, one can apply
   \cite[Proposition 4.3]{HHH} to $X(Q, \lambda)$.
\begin{figure}[ht]
        \centerline{
           \scalebox{0.42}{
            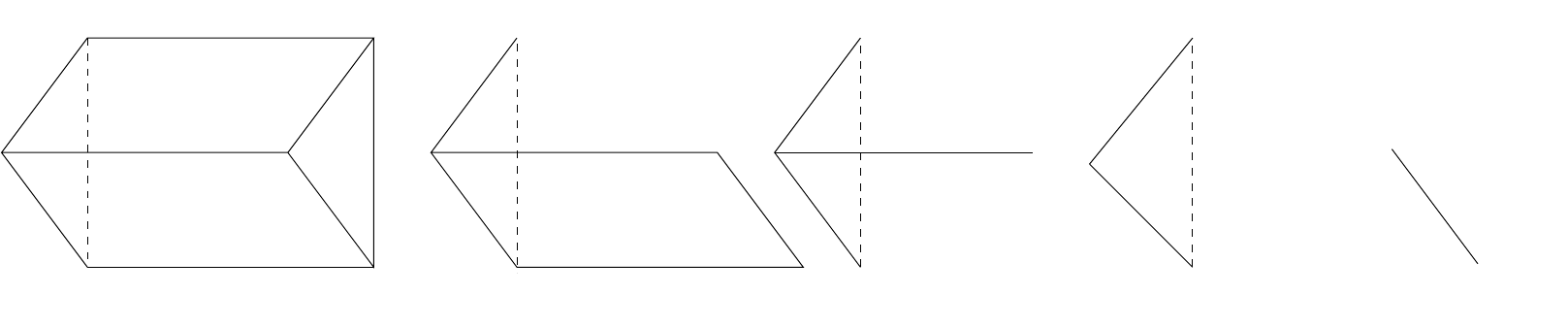
            }
          }
 \caption{A retraction sequence of 3-prism $Q$.}
 \label{fig-kthm3}
 \end{figure}  
\end{remark}

\section{GKM theory on retractable toric orbifolds}\label{sec_gkm-theory_toric_orb}

We begin this section by recalling the GKM theory from \cite[Section
3]{HHH}. We shall then verify that these results can be applied to a
retractable toric orbifold and hence give a precise description of its
equivariant generalized cohomology theory ring.

Let $X$ be a $G$-space equipped with a $G$-invariant stratification
\[X_{m} \subseteq X_{m-1} \subseteq \cdots \subseteq X_1 = X.\] That
is, $X_i\setminus X_{i+1}$ has a $G$-invariant subspace ${Y}_i$ having
a $G$-stable neighbourhood which is homeomorphic to the total space
$V_i$ of a $G$-equivariant vector bundle $\rho_i=(V_i,\varpi_i,Y_i)$ with
projection $\varpi_i \colon V_i\lra Y_i$. In particular, when $Y_i=x_i$ is a
$G$-fixed point then $\rho_i=(V_i,\varpi_i,x_i)$ is a $G$-representation.

Let $E_{G}^*$ be a generalized $G$-equivariant cohomology
theory. We now make the following assumptions on $X$.

\begin{enumerate}
\item[(A1)] Each subquotient $X_i/X_{i+1}$ is homeomorphic to the Thom
  space $Th(\rho_i)$ with corresponding attaching map
  $\phi_i \colon S(\rho_i)\longrightarrow X_{i+1}$.

\item[(A2)] Every $\rho_i$ admits a $G$-equivariant direct sum
  decomposition $\displaystyle\bigoplus_{j>i} \rho_{ij}$ into $G$-equivariant
  subbundles $\rho_{ij}=(V_{ij},\varpi_{ij}, {Y}_i)$. We allow the case $V_{ij}=0$.

\item[(A3)] There exist $G$-equivariant maps  $f_{ij} \colon {Y}_i
  \longrightarrow {Y}_j$ such that the restrictions $f_{ij}\circ
  \varpi_{ij}\mid_{S(\rho_{ij})}$ and 
$\phi_i\mid_{S(\rho_{ij})}$ agree for every $j > i$. 

\item[(A4)]The equivariant Euler classes $e_G(\rho_{ij})$ for $j > i$, are not
   zero divisors and are pairwise relatively prime in $E^*_G({Y}_i)$. 
\end{enumerate}

We now recall the precise description of the generalized
$G$-equivariant cohomology ring of $X$.

\begin{theorem}\cite[Theorem 3.1]{HHH}\label{gkm}
  Let $X$ be a $G$-space satisfying the four assumptions {\rm (A1)} to
  {\rm (A4)}. Then the restriction map
\[\iota^* \colon E^*_{G}(X)\longrightarrow \prod_{i=1}^m E_{G}^*({Y}_i)\] is
monic and its image $\Gamma_{X}$ can be described as
\[\{(a_i)\in \prod_{i=1}^m E_G^*({Y}_i)~~:~~\mbox{\em for every}~ j>i,~~
e_{G}(\rho_{ij})\mid a_i-f_{ij}^*(a_j)\}.\]
\end{theorem}

We show below that the GKM theory of \cite{HHH} described above can be
applied to a retractable toric orbifold $X :=X(Q,\lambda)$.  We further
use this to give an explicit description of the $T$-equivariant
$K$-theory ring $K_T^{*}(X)$, the $T$-equivariant complex cobordism
ring $MU^*_{T}(X)$ and the $T$-equivariant integral cohomology ring
$H^*_T(X)$ (see Corollary \ref{gkm_rt_algebra}).

Let $X$ be a retractable toric orbifold and
$\{(B_i, P_i, v_i)\}_{i=1}^{m}$ be the corresponding retraction
sequence of the polytope $Q$.  By Lemma \ref{cw-of-toic-orbi} there is
a $T$-invariant retractable structure on $X$ associated to this
retraction.  Let
$$X_i:= \pi^{-1} (B_{i}), \quad \mbox{and let} \quad x_i :=
\pi^{-1}(v_i)$$ for $1 \leq i \leq m$. Then $x_1,\ldots, x_{m-1}$ and $x_m$
are the $T$-fixed points of $X$. Let
$U_i\cong \mathbb{R}_{\geq}^{s_i}$ denote the open neighbourhood of
$v_i$ in $P_i$ obtained by deleting all faces of $P_i$ not containing
the free vertex $v_i$ and let $\widehat{U}_i:=\pi^{-1}(U_i)$. Thus we
have the following $T$-invariant stratification
\begin{equation}\label{T-inv_stra}
\{x_m\} = X_{m} \subseteq X_{m-1} \subseteq \cdots \subseteq X_1 = X
\end{equation}
of $X$ such that
$X_i \setminus X_{i+1} = \widehat{U}_{i} \cong \bb C^{s_i}$, where
$s_i = \dim P_{i}$ for $i=1, \ldots, m$ with $X_{m+1}=
\emptyset$. Moreover, $x_i \in \widehat{U}_i \subseteq X_i$,
$1 \leq i \leq m$. Since $\widehat{U}_{i} \cong \mathbb{C}^{s_i}$ are
$T$-stable we have a $T$-representation $\rho_i = ({V}_i, \varpi_i, x_i)$
for $1 \leq i \leq m$. Here $T$ acts on ${V}_i:=\widehat{U}_{i}$ via
its projection to $T^{s_i}$ given by the characters
$u_{j_1},\ldots, u_{j_{s_i}}$ (described in the proof of Proposition
\ref{divass}) followed by the standard action of $T^{s_i}$ on
$\mathbb{C}^{s_i}$.

\begin{proposition}\label{divass}
  A retractable toric orbifold $X$ with the 
  $T$-invariant stratification in \eqref{T-inv_stra} satisfies assumptions (A1) to
  (A4) listed above. 
\end{proposition}

\begin{proof} {\bf Checking for (A2)}: Let
  $P_i=F_{i_1}\cap\cdots\cap F_{i_{n-s_i}}$. Consider the
  $\mathbb{Z}$-linear
  map \begin{equation}\label{char_map}\psi \colon \mathbb{Z}^n\lra
    \mathbb{Z}^{n-s_i}\end{equation} defined by the
  $(n-s_i)\times n$-matrix with rows
  $\lambda(F_{i_1}),\ldots, \lambda(F_{i_{n-s_i}})$. We note that $T$
  acts on $V_i$ via the characters $u_{j_1},\ldots, u_{j_{s_i}}$ which
  form a $\mathbb{Z}$-basis of the kernel of $\psi$. In other words,
  $V_i$ is a direct sum of the one-dimensional representations
  $V_{ij_{r}}:=\mathbb{C}_{u_{j_r}}$, $1\leq r\leq s_i$ of $T$. Recall
  that for $1\leq r\leq s_i$, there is an edge $e_{j_r}$ in the
  directed graph associated to the retraction sequence which points
  towards $v_i$. Then $u_{j_r}$ is a $\mathbb{Z}$-basis of the kernel
  of the $\mathbb{Z}$-linear map
  $$\psi_r \colon \mathbb{Z}^n\lra \mathbb{Z}^{n-1}$$ defined by an
  $(n-1)\times n$-matrix which has a row $\lambda(F)$ corresponding to
  every facet $F$ of $Q$ containing $e_{j_r}$. Note that $u_{j_r}$ is
  unique up to a sign. Since $T$ acts on the invariant subspace
  $\pi^{-1}(e_{j_r})$ via the character $u_{j_r}$, the sign of
  $u_{j_r}$ for $1\leq r\leq s_i$ is determined by the $T$-action on
  $X$.  It follows that $V_{ij_r}$ corresponds to the directed edge
  $e_{j_r}$ for $1\leq r\leq s_i$. By putting $V_{ik}=0$ for $k>i$ and
  $k\notin\{j_1,\ldots, j_{s_i}\}$ we see that {\bf (A2)} follows. We
  denote by $\rho_i$ the representation $(V_i,\varpi_i,x_i)$ and by
  $\rho_{ij_r}=(V_{ij_r},\varpi_i\mid_{V_{ij_r}}, x_i)$ the
    one-dimensional sub-representation associated to the character
    $u_{ij_r}$ for $1\leq r\leq s_i$.\\

  {\bf Checking for (A1)}: Since
  $X_{i}\setminus X_{i+1}=\widehat{U}_i=V_i$ is a complex
  $T$-representation where $T$ acts by characters $u_{ij_r}$ for
  $1\leq r\leq s_i$, it follows that $X_{i}/X_{i+1}$ is the
  representation sphere $S(\rho_i)$. Here $S(\rho_i)$ is the
  equivariant one point compactification of $V_i$ with infinity viewed
  as a fixed point. This is, therefore, nothing but the Thom space
  $Th(\rho_i)$ of the vector bundle $V_i$ over the point $x_i$. We now
  describe in detail the attaching map $$\phi_i \colon S(\rho_i)\lra X_{i+1}.$$

  Consider the neighbourhood $W_i=U_i \cap D$ of $v_i$ in $Q$ where
  $U_i$ is the open neighbourhood of $v_i$ in $P_i$ obtained by
  deleting all faces of $P_i$ not containing the free vertex $v_i$ and
  $D$ is a closed disc in $\mathbb{R}^n$ with centre $v_i$ such that
  $D$ does not contain any other vertex of $Q$. Note that $P_i$ is the
  smallest face of $Q$ containing the edges $e_{j_r}$ for
  $1\leq r\leq s_i$. The $\mbox{Link}(v_i)$ in $P_i$ is
  $P_i \setminus U_i$. So
  $P_i = \mbox{Star}(v_i) = \mbox{Link}(v_i)\star v_i$. Thus it
  follows from polyhedral geometry that for every
  $p \in (P_i \setminus U_i)$, the line segment joining $p$ and $v_i$
  meets $W_i \cap \partial D $ at a unique point $y_p$. Moreover,
  $y_p$ determines $p$ uniquely and vice versa, see
  Figure \ref{fig-a1} for an example.
\begin{figure}[ht]
  \centerline{ \scalebox{.70}{ 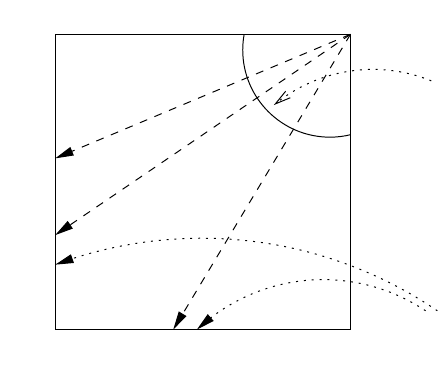 } }
 \caption{Correspondence of cell attaching.}
\label{fig-a1}
 \end{figure} 

This gives a bijective correspondence $g_i \colon \partial{D} \cap W_i \to P_i \setminus U_i$.
Therefore we have following commutative diagram,
\begin{equation}\label{def_attaching}
\begin{CD}
  T \times (\partial{D} \cap W_i) @>id \times g_i>> T \times (P_i \setminus U_i) @.\\
  @VVV @VVV  @. \\
  (T \times (\partial{D} \cap W_i))/\sim @>{\phi_i}>> (T \times
  (P_i \setminus U_i))/\sim @>{\subseteq}>> X_{i+1}.
\end{CD}
\end{equation} 
The map $\phi_i$ sends $[t,y_p]$ to $[t,p]$. This map is
well defined because if $y_p$ belongs to the relative interior of a face $F$
of $P_i$ then $p$ also belongs to $F$.  Moreover, under the identification of 
$\widehat{U}_i$ with the complex representation $\rho_i$,
$$\widehat{W}_i = \pi^{-1}(W_i)= (T \times W_i)/\sim ~ \subseteq \widehat{U}_i$$
can be identified with the disc bundle $D(\rho_{i})$ and
$\partial(\widehat{W}_i) = \pi^{-1}(\partial{D} \cap W_i))$ with the
sphere bundle $S(\rho_i)$ of the representation space associated with
$\rho_i$. The above identifications induce the following
homeomorphisms
$$X_i/X_{i+1} ~\cong ~ \widehat{W}_i/\partial(\widehat{W}_i) ~\cong ~
D(\rho_i)/S(\rho_i) = Th(\rho_i),$$ where $\widehat{U}_i \subseteq X_i$ maps
homeomorphically onto the interior of  $\widehat{W}_i$. This verifies
{\bf (A1)}.\\

{\bf Checking for (A3)}: Let the initial vertex of $e_{j_r}$ be
$v_{j_r}$ and let $f_{ij_r} \colon x_i \longrightarrow x_{j_r}$ denote the
constant map for $1\leq r\leq s_i$. Further,
$\pi^{-1}(e_{j_r} \setminus v_{j_r}) \subseteq \widehat{U}_i$ can be
identified with the one dimensional sub-representation $\rho_{ij_r}$
of $\rho_i$ for $r=1, \ldots, s_i$. Let $S(\rho_{ij_r})$ denote the
circle bundle associated with $\rho_{ij_r}$. Let $w_{j_r}$ be a point
where $e_{j_r}$ meets $\partial{D} \cap W_i$. Then the attaching map
$\phi_i\mid_{S(\rho_{ij_r})}$ (see \eqref{def_attaching}) sends
$\pi^{-1}(w_{j_r})$ in $S(\rho_{ij_r})$ to $x_{j_r}$. Further,
$\pi^{-1}(w_{j_r}) \in \widehat{U}_i$ is mapped to $x_i$ under the
canonical projection in the radial direction of the representation
$\rho_{ij_r}$.  It follows that the restriction of the map $\phi_i$
and the composition of the projection of $\rho_{ij_r}$ with $f_{ij_r}$
agree on $S(\rho_{ij_r})$ for every $1\leq r\leq s_i$.
This verifies assumption {\bf (A3)}.\\

{\bf Checking for (A4)}: Recall that $u_{j_r}$ for $1 \leq r\leq s_i$
are a $\mathbb{Z}$-basis for the kernel of $\psi$ (see
\ref{char_map}). In particular, for $1 \leq r\leq s_i$, $u_{j_r}$ is a
primitive non-zero vector. Hence the $K$-theoretic equivariant Euler
class \[e^{T}(\rho_{ij_r})=1-e^{-u_{j_r}}\] is a non-zero divisor in
the integral domain $K^0_T(x_i)=R(T)$. Also, $u_{j_r}$ for
$1 \leq r\leq s_i$, are pairwise linearly independent. Thus
$1-e^{-u_{j_r}}$ and $1-e^{-u_{j_{r'}}}$ are relatively prime in the
unique factorization domain $R(T)$ for $r \neq r'$ .  This can also be
seen more generally for the equivariant Euler classes in $MU^*_{T}$
and also for $H^*_{T}( ;\mathbb{Z})$ (see \cite[Lemma 5.2]{HHH}).
\end{proof}

\subsection{Equivariant generalized cohomology theory of retractable
  toric orbifolds}
In this subsection, we describe the $T$-equivariant generalized
cohomology ring of a retractable toric orbifold $X$ as an
$E_{T}^*(pt)$-algebra.

The following corollary extends \cite[Proposition 3.10]{HHRW} on a
divisive weighted projective space and \cite[Theorem 7.1]{HHRW} on the
toric variety associated with a smooth polytopal fan to any retractable
toric orbifold.  Consider
$f_{i1}^* \colon E^*_{T}(x_m)\longrightarrow E^*_{T}(x_i)$ induced by the
constant map $f_{i1} \colon x_i\lra x_m$ for $1\leq i\leq m$.  This gives
$\displaystyle\prod_{i=1}^m E^*_{T}(x_i)$ a canonical $E^*_{T}(x_m)$-algebra
structure via the inclusion defined by $(f_{i1}^*(a))$ for
$a\in E^*_{T}(x_m)$.

Let $V_{ij_r}$ denote the $1$-dimensional $T$-representation
corresponding to the primitive character $u_{j_r}\in
\mathbb{Z}^n$. When there is no edge between $v_i$ and $v_k$ for $k>i$
then $V_{ik}$ is trivial.

\begin{corollary}\label{gkm_rt_algebra}
  Let $X=X(Q,\lambda)$ be a retractable toric orbifold. The
  $T$-equivariant generalized cohomology theory ring $E_T^*(X)$, for
  $E=K, MU, H$ is isomorphic to the $E_{T}^*(x_m)$-subalgebra
\begin{equation}\label{gkm_desc_kring}
\Gamma_{X}=\{(a_i)~:~\forall~k>i,~~ e^{T}({V_{ik}})~\mbox{\em divides}~
 a_i-f^*_{ik}(a_{k}) \in E^*_{T}(x_i) \}\subseteq 
 \prod_{i=1}^m E_{T}^*(x_i).
\end{equation} (Here $H$ denotes integral cohomology.)
\end{corollary} 
\begin{proof}
  By Proposition \ref{divass} and Theorem \ref{gkm} above it follows
  that $E^*_{T}(X)$ is isomorphic to the subring $\Gamma_{X}$ of
  $\displaystyle\prod_{i=1}^m E_{T}^*(x_i)$. Note that,
  $\forall~~a\in E^*_{T}(x_1)$, $(f_{i1}^*(a))\in \Gamma_{X}$.  This
  is because for every $k>i$,
  $f^*_{i1}(a)-f^*_{ik}(f^*_{k1}(a))=f^*_{i1}(a)-f^*_{i1}(a)=0$ and is
  hence trivially divisible by $e^{T}({V_{ik}})$. Thus $\Gamma_{X}$ is
  an $E^*_{T}(x_m)$-subalgebra of
  $\displaystyle\prod_{i=1}^m E^*_{T}(x_i)$.
\end{proof}

\section{Piecewise algebra and its applications}\label{sec_piecewise_alg}
In this section, we introduce the concept of piecewise algebra
associated to the characteristic pair $(Q,\lambda)$ in a dual manner
to the notion of piecewise algebra associated to a fan (see
\cite[Section 4]{HHRW}). Consider the category $\mathrm{Face}(Q)$
whose objects are the faces $F$ of $Q$ and whose morphisms are their
inclusions $i_{F,F'} \colon F \hookrightarrow F'$.  Then
$\mathrm{Face}(Q)$ is a small category in which $Q$ is the final
object.

We have a covariant functor $\mathcal{E}_{T}^*\mathrm{-CGA}$ from
$\mathrm{Face}(Q)$ to the category of graded commutative
$E_T^*(pt)$-algebras defined by sending a face $F$ of $Q$ to
$$\mathcal{E}_{T}^*\mathrm{-CGA} (F):=E_T^*(T/T_F)$$ and the inclusion $i_{F,F'}$ to the $E_{T}^*(pt)$-algebra
morphism $i_{F,F'}^* \colon E_T^*(T/T_F)\lra E_T^*(T/T_{F'})$ induced by the
inclusion $T_{F'}\hookrightarrow T_F$ (see \cite[Definition 4.3,
4.4]{HHRW}).

\begin{definition}\label{piec_alg}
  The limit $\mbox{\em lim}~~\mathcal{E}_{T}^*\mathrm{-CGA}$ is called the
  {\em $E_{T}^*(pt)$-algebra of piecewise $E_T^*$-coefficients for the
  characteristic pair $(Q, \lambda)$} denoted by $\mathcal{P}_{E}(Q, \lambda)$.
\end{definition}

We note that here $\mathcal{P}_{E}(Q, \lambda)$ is an $E_{T}^*(pt)$-subalgebra of
$\displaystyle\prod_{F} E^*_{T}(T/T_{F})$, so every piecewise coefficient has one
component $f_{F}$ for every face $F$ of $Q$. For $1\leq i\leq m$ we have
$$\mathcal{E}^*_{T}\mathrm{-CGA} (v_i)= E_{T_{v_i}}^*(pt) = E_T^* (pt)$$
since $T_{v_i}=T$; on the other hand 
$$\mathcal{E}^*_{T}\mathrm{-CGA} (Q) = E_{T_Q}^*(pt)$$ since $T_{Q}=\{1\}$. 
Moreover, if $(f_{F})\in \mathcal{P}_{E}(Q, \lambda)$ then
$i_{F,F'}^*(f_{F})=f_{F'}$ whenever $F\subseteq F'$ in $Q$ which is
called the {\em compatibility} condition. Sums and products of
piecewise coefficients are take facewise. We have a canonical diagonal
inclusion $E_{T}^*(pt) \subseteq \mathcal{P}_{E}(Q, \lambda)$ as $(i^*_{F} (f))$, for
$f\in E_{T}^*$ where $i^*_{F} \colon E_T^*(pt) \lra E_T^*(T/T_F)$ is induced by
the projection $$T/T_{F}\lra T/T=pt$$ associated to the canonical
inclusion $T_{F}\subseteq T$, which clearly satisfies the
compatibility condition. The image of the diagonal inclusion is the
subalgebra of {\em global coefficients}. Also, the constants $(0)$ and
$(1)$ act as the zero and identity element in $\mathcal{P}_{E}(Q, \lambda)$
respectively (see \cite[Remark 4.8]{HHRW}).

Let $F$ be a face of $Q$ of codimension $n-k$. Let $F$ be the
intersection of the facets $F_1, F_2,\ldots, F_{n-k}$. Then consider
the $\mathbb{Z}$-linear map
\[\psi_{F} \colon \mathbb{Z}^n \longrightarrow \mathbb{Z}^{n-k},\] defined by
the $(n-k)\times n$-matrix with rows
$\lambda(F_1), \ldots, \lambda(F_{n-k})$. The kernel of $\psi_{F}$ is
a free $\mathbb{Z}$-module generated by the primitive vectors
$u_{1},\ldots, u_{k}$ in $\mathbb{Z}^n$. Since $u_{i}$'s are pairwise
linearly independent, $e^{T}({u_i}), 1\leq i\leq k$ are relatively
prime in $E^*_{T}(pt)$ for $E=K,MU,H$, see proof of {\bf (A4)} in
Proposition \ref{divass}.  (Here $e^{T}(u_i)\in E^*_{T}(pt)$ denotes the
$T$-equivariant Euler class of the $1$-dimensional $T$-representation
corresponding to $u_i$.) Thus as in \cite[Example 4.12]{HHRW} we have
the isomorphism
\begin{equation}\label{isom} E^*_{T}(pt)/J_{F} \cong
  E^*_{T}(T/T_{F}).
 \end{equation} 
 where $ J_{F}$ is an ideal of $E_{T}^*(pt)$ generated by
 $e^T({u_{i}}), 1 \leq i \leq k$.

 Further, if $F\subseteq F'$ then
 $\mbox{kernel}(\psi_{F})\subseteq \mbox{kernel}(\psi_{F'})$. This
 gives the inclusion of the ideals $J_{F}\subseteq J_{F'}$ of
 $E^*_{T}(pt)$ inducing the projection
$$r^*_{F,F'} \colon E^{*}_{T}(pt)/J_{F}\longrightarrow E^*_{T}(pt)/J_{F'},$$ which
corresponds to $i^*_{F,F'}$ under the identification (\ref{isom}).

The following theorem is an extension of \cite[Theorem 5.5]{HHRW} for
a divisive weighted projective space and \cite[Corollary 7.2]{HHRW}
for a toric variety corresponding to a smooth polytopal fan to any
retractable toric orbifold. The proof follows closely that of
\cite[Theorem 5.5]{HHRW} suitably adapted to this setting.

\begin{theorem}\label{piecewise}
  For a retractable toric orbifold $X=X(Q,\lambda)$,
  $E^*_{T}(X(Q,\lambda))$ is isomorphic as an $E^*_{T}(pt)$-algebra to
  $\mathcal{P}_{E}(Q, \lambda)$ for each $E=K,MU,H$. In the case when
  $E=H$, $\mathcal{P}_{E}(Q, \lambda)$ is the ring of piecewise
  polynomial functions, when $E=K$, it is the ring of piecewise
  Laurent polynomial functions and when $E=MU$, it is the ring of
  piecewise cobordism forms on $(Q,\lambda)$.
\end{theorem}

\begin{proof}
  It suffices to identify the algebra $\Gamma_{X}$ defined by
  (\ref{gkm_desc_kring}) with $\mbox{lim}~~\mathcal{E}^*_{T}\mathrm{-CGA}$. By the
  universal property of $\mbox{lim}~~\mathcal{E}^*_{T}\mathrm{-CGA}$, we first find
  compatible homomorphisms
  $$h_{F} \colon \Gamma_{X}\longrightarrow \mathcal{E}^*_{T}\mathrm{-CGA}(F)$$ for every
  face $F$ of $Q$.  If $a=(a_i)\in \Gamma_{X}$, on the vertex $v_k$
  of $Q$ we define $h_{v_k}(a):=a_k$ for each $1\leq k\leq m$. On an
  edge $e_{ij}$ joining $v_i$ and $v_j$ we let
  $$ h(a)_{e_{ij}}:=a_{i} \pmod {J_{e_{ij}}}~\in~ E^*_{T}(pt)/J_{e_{ij}}.$$ This is well defined since
  $e^{T}(u_{ij})~\mid~a_i-a_j$, where $u_{ij}$ generates
  $\mbox{kernel}(\psi_{e_{ij}})$, and $J_{e_{ij}}$ contains
  $e^{T}(u_{ij})$.  For any face $F$ of $Q$ having vertices
  $v_{i_1},\ldots, v_{i_k}$ we let
$$ h_{F}(a):=a_{i_1}\pmod {J_{F}}~\in ~E_{T}^*(pt)/J_{F}.$$ Since $u_{{ij}}$ generates
$\mbox{kernel}(\psi_{e_{ij}})\subseteq \mbox{kernel} (\psi_{F})$,
$J_{F}$ contains $e^{T}(u_{ij})$ for each edge $e_{ij}\subset F$ and
hence contains $a_i-a_j$. Further, since the $1$-skeleton of $F$ is
connected, any two vertices $v_{i_{j}}$ and $v_{i_{r}}$ are connected
by a path of edges in $F$. It follows that $a_{i_j}-a_{i_r}$ belongs
to $J_{F}$, and therefore the map $h_{F}$ is well defined for each
face $F$ of $Q$. Furthermore, since $J_{F}$ is an ideal in $E^*_{T}(pt)$,
it follows that $h_{F}$ is a homomorphism of $E^*_{T}(pt)$-algebras.

Moreover, $h_{F}$'s are compatible over $\mbox{Face}(Q)$.
This follows as $F \subseteq F'$ implies $J_{F'}$ is obtained from
$J_{F}$ by adjoining $e^{T}({u})$ for
$\displaystyle u \in \mbox{kernel}(\psi_{F'})\setminus
\mbox{kernel}(\psi_{F})$. Thus the corresponding projection
$$r^*_{F,F'} \colon E^*_{T}(pt)/J_{F}\longrightarrow E^*_{T}(pt)/J_{F'}$$ satisfies
$h_{F'}=r^*_{F,F'} \circ h_{F}$ whenever $F \subseteq
F'$. Therefore we have constructed a well defined homomorphism
$$ h \colon \Gamma_{X}\longrightarrow \mathcal{P}_E(Q, \lambda)$$ of
$E^*_{T}(pt)$-algebras.

We now conclude by showing that the map $h$ is an isomorphism. Given
$\displaystyle a\neq a' \in \Gamma_{X}\subseteq \prod_{i=1}^m
E^*_{T}(x_i)$. There exists at least one $v_i$ such that
$a_i\neq a'_i$. Thus $h_{v_i}(a)\neq h_{v_i}(a')$ in
$E^*_{T}(x_i)$. Hence $h$ is injective.  Let $(a_{F})$ be an element
in the limit $\mathcal{P}_E(Q, \lambda)$ of the functor
$\mathcal{E}^*_{T}$. Then $(a_F)$ determines $(a_i)$ in
$\displaystyle\prod_{i=1}^m E^*_{T}(x_i)$ by restricting to the
vertices $a_i:=a_{v_i}\in E^*_{T}(pt)$ for $1\leq i\leq m$. Whenever $v_i$
and $v_j$ are connected by an edge $e_{ij}$ in $Q$, we have
$$a_{e_{ij}} = r^*_{v_i, e_{ij}}(a_{v_i}) \quad \mbox{and}
 \quad a_{e_{ij}} = r^*_{v_j, e_{ij}}(a_{v_j})$$ and hence
${\displaystyle a_{i}\equiv a_{j} \pmod {J_{e_{ij}}}}$. Since
$J_{e_{ij}}$ is generated by $e^{T}({u_{ij}})$ it follows that
$a_i-a_j$ is divisible by $e^{T}({u_{ij}})$. This implies by
(\ref{gkm_desc_kring}) that $(a_i)\in \Gamma_{X}$ proving the surjectivity of
$h$.
\end{proof}

We now illustrate the above theorem by describing the
equivariant generalized cohomology ring of the retractable toric manifold
of Example \ref{eg_div_toric_3d}, as a piecewise polytopal algebra.

\begin{example} The edges $e_{ij}$ in the polytope $Q$ given in Figure
  \ref{fig-eg1} are listed below:
  \[\begin{matrix}e_{13}=F_1\cap F_3 & e_{12}=F_1\cap F_2 &
      e_{23}=F_1\cap F_4\\e_{14}=F_3\cap F_2 & e_{36}=F_3\cap F_4 &
      e_{45}=F_2\cap
                                                                    F_5\\
      e_{46}=F_3\cap F_5& e_{56}=F_4\cap F_5 &e_{25}=F_2\cap
                                               F_4 \end{matrix}\] We
                                               list below the
                                               corresponding
                                               $\psi_{e_{ij}}$ and
                                               $u_{ij}$.
                                                             \[\begin{matrix}
                                                                 \psi_{e_{13}}=\begin{bmatrix}
                                                                                      1 &0&0\\0&0&1\end{bmatrix}&
                                                                                                                         \psi_{e_{12}}=\begin{bmatrix} 1&0&0\\0&1&0\end{bmatrix}&
                                                                                                                                                                                 \psi_{e_{23}}=\begin{bmatrix} 1&~~0&0\\6&-1&3\end{bmatrix}\\\\
                                                                                                                                                                                 \psi_{e_{14}}=\begin{bmatrix} 0&0&1\\0&1&0\end{bmatrix} & \psi_{e_{36}}=\begin{bmatrix}0&~~0&1\\6&-1&3\end{bmatrix} &
                                                                                                                                                                                                                                                                                                                                                           \psi_{e_{45}}=\begin{bmatrix}0&1&0\\1&1&4\end{bmatrix}\\\\
                                                                 \psi_{e_{46}}=\begin{bmatrix} 0&0&1\\1&1&4\end{bmatrix}&
                                                                                                                         \psi_{e_{56}}=\begin{bmatrix}6&-1&3\\1&~~1&4 \end{bmatrix}
                                                          
                                                                                                                 &\psi_{e_{25}}=\begin{bmatrix}0&~~1&0\\6&-1&3\end{bmatrix}
                                                                 \end{matrix}\]
                                                             
\[  \begin{matrix}
                                      u_{13}=(0,1,0) &
                                                           u_{12}=(0,0,1)
                                                            &
                                                                 u_{23}=(0,3,1)
                                                                 \\u_{14}=(1,0,0)
                                       & u_{36}=(1,-6,0)&
                                                                 u_{45}=(-4,0,1)
                                                                 \\
                                      u_{46}=(-1,1,0)&
                                                          u_{56}=(1,3,-1)
                                                          
                                                               &u_{25}=(1,0,-2)
                                                                 \end{matrix}\]

                                                               Then
                                                               $E_{T}^*(X)$
                                                               consists
                                                               of
                                                               $(a_i
                                                               )\in
                                                               \displaystyle\prod_{i=1}^6
                                                               E_{T}^*(x_i)$
                                                               satisfying
                                                               the
                                                               following
                                                               relations:

 \[\begin{matrix} a_{1}- a_{3}&\equiv 0 \pmod {e^{T}(0,1,0)} & a_{4}- a_{6}&\equiv 0 \pmod {e^T(-1,1,0)}
\\a_{1}- a_{4} &\equiv 0\pmod {e^{T}(1,0,0)} &
                                                               a_{3}-
                                                                                             a_{6}
                                                                                             &\equiv
                                                                                             0\pmod
                                                                                             {e^T(1,-6,0)}
                                                                                                           & \\
a_{1}- a_{2} &\equiv0\pmod {e^{T}(0,0,1)}  & a_{4}- a_{5} &\equiv
                                                0\pmod {e^T(-4,0,1)}
                                                                                                           &\\
     a_{2}- a_{3}&\equiv 0 \pmod {e^{T}(0,3,1)}& a_{5}- a_{6} &\equiv
                                                0\pmod {e^T(1,3,-1)}

\\ a_{2}-
                                                                                                             a_{5}
                                                                                                             &\equiv
                                                                                                             0\pmod
                                                                                                             {e^T(1,0,-2)}&
                                                \end{matrix}\]

\end{example}

\begin{remark}
\begin{enumerate}
\item In Section 4.2 of \cite{BNSS}, the authors constructed the characteristic
pair corresponding to a polytopal simplicial complex. So one can define retractable
toric variety using this characteristic pair following Definition
 \ref{def_divisive_tor_orb}. Therefore, as a consequence we can get similar
 description of the equivariant generalized cohomology
theories for retractable toric varieties arising from polytopal simplicial complexes.\\

\item In \cite{DJ}, Toric manifolds were studied and an invariant CW-structure
of a toric manifold was constructed. So in particular, toric manifolds are
retractable toric orbifolds, and hence Theorem \ref{piecewise} holds for this 
class of manifolds. See \cite{DKU} for the description of the equivariant
$K$-ring of toric manifolds as a Stanley-Reisner ring.\\

\item In subsection 4.3 of \cite{BNSS}, the local groups $G_F(v)$ are computed
for torus orbifolds which are generalizations of toric orbifolds. So 
Definition \ref{def_divisive_tor_orb} can be introduced in this category. 
Thus one can get similar description of the equivariant generalized cohomology
theories for retractable torus orbifolds.

\item In a recent related paper \cite{HW}, the authors address the
  following question (see \cite[Question 1.4]{HW}): for which fans the
  $T$-equivariant $K$-theory ring of the associated toric variety is
  isomorphic to the ring of piecewise Laurent polynomial functions on
  the fan. In particular in \cite[Theorem 7.2]{HW} they show that fans
  with \emph{distant singular cones} satisfy this property. Our result
  Theorem \ref{piecewise} shows in particular that simplicial
  polytopal fans whose associated toric varieties are retractable
  toric orbifolds with respect to the action of the compact torus
  satisfy this property.

\item In the paper \cite{LT} the authors study symplectic toric
  orbifolds. In particular, in \cite[Section 9]{LT} they show that
  every symplectic toric orbifold has the structure of a toric variety
  associated to the fan dual to the corresponding moment polytope and
  is hence a simplicial projective toric variety. The symplectic toric
  orbifold is therefore retractable if the associated moment polytope
  admits a retraction sequence satisfying the conditions of Definition
  \ref{def_divisive_tor_orb}.

\end{enumerate}
\end{remark}
  
{\bf Acknowledgement}: The authors would like to thank Anthony Bahri,
Nigel Ray and Jongbaek Song for helpful conversations. The authors
are grateful to the referee for valuable comments and suggestions for
improvement of the manuscript.
  


\end{document}